\documentclass{amsart}
\usepackage[all]{xy}
\usepackage{color}
\usepackage{amsthm}
\usepackage{amssymb}
\usepackage[colorlinks=true]{hyperref}
\usepackage{tikz}
\usetikzlibrary{matrix,arrows}

\numberwithin{equation}{section}

\newtheorem{theorem}[equation]{Theorem} 
\newtheorem*{theorem*}{Theorem}
\newtheorem{lemma}[equation]{Lemma}
\newtheorem{proposition}[equation]{Proposition}
\newtheorem{corollary}[equation]{Corollary}
\newtheorem*{corollary*}{Corollary}

\theoremstyle{remark}

\newtheorem{notation}[equation]{Notation}

\theoremstyle{remark}
\newtheorem{remark}[equation]{Remark}

\newcommand{\cA}{{\mathcal A}}
\newcommand{\cB}{{\mathcal B}}
\newcommand{\cC}{{\mathcal C}}
\newcommand{\cD}{{\mathcal D}}

\newcommand{\cM}{{\mathcal M}}

\newcommand{\cO}{{\mathcal O}}

\newcommand{\cU}{{\mathcal U}}

\newcommand{\bbA}{\mathbb{A}}

\newcommand{\bbL}{\mathbb{L}}

\newcommand{\bbN}{\mathbb{N}}
\newcommand{\bbP}{\mathbb{P}}

\newcommand{\bbZ}{\mathbb{Z}}

\DeclareMathOperator{\id}{id}

\DeclareMathOperator{\Mot}{Mot}

\newcommand{\dgcat}{\mathsf{dgcat}}

\newcommand{\perf}{\mathsf{perf}}

\newcommand{\dg}{\mathsf{dg}}

\newcommand{\Hom}{\mathrm{Hom}}

\newcommand{\Ho}{\mathsf{Ho}}
\newcommand{\HO}{\mathsf{HO}}

\newcommand{\too}{\longrightarrow}

\newcommand{\ie}{\textsl{i.e.}\ }
\newcommand{\eg}{\textsl{e.g.}}

\begin{document}

\title[$E_n$-regularity implies $E_{n-1}$-regularity]{$E_n$-regularity implies $E_{n-1}$-regularity}

\author{Gon{\c c}alo~Tabuada}

\address{Gon{\c c}alo Tabuada, Department of Mathematics, MIT, Cambridge, MA 02139, USA}
\email{tabuada@math.mit.edu}
\urladdr{http://math.mit.edu/~tabuada/}
\thanks{The author was partially supported by the NEC Award-2742738.}
\date{\today}


\abstract{Vorst and Dayton-Weibel proved that $K_n$-regularity implies $K_{n-1}$-regularity. In this article we generalize this result from (commutative) rings to differential graded categories and from algebraic $K$-theory to any functor which is Morita invariant, continuous, and localizing. Moreover, we show that regularity is preserved under taking desuspensions, fibers of morphisms, direct factors, and arbitrary direct sums. As an application, we prove that the above implication also holds for schemes. Along the way, we extend Bass' fundamental theorem to this broader setting and establish a Nisnevich descent result which is of independent interest.}}

\maketitle 
\vskip-\baselineskip
\vskip-\baselineskip
\section{Introduction}\label{sec:introduction}
Let $n \in \bbZ$. Following Bass \cite[\S XII]{Bass}, a (commutative) ring $R$ is called {\em $K_n$-regular} if $K_n(R)\simeq K_n(R[t_1,\ldots, t_m])$ for all $m \geq 1$. The following~implication
\begin{equation}\label{eq:implication}
R \,\, \mathrm{is} \,\, K_n\textrm{-}\mathrm{regular} \Rightarrow R \,\, \mathrm{is} \,\, K_{n-1}\textrm{-}\mathrm{regular}
\end{equation}
was proved by Vorst~\cite[Cor.~2.1]{Vorst} for $n \geq 1$ and latter by Dayton-Weibel~\cite[Cor.~4.4]{Hyperplanes} for $n \leq 0$. It is then natural to ask the following:

\vspace{0.1cm}

{\it Question: Does implication \eqref{eq:implication} holds more generally ?}
\subsection*{Statement of results}
A {\em differential graded (=dg) category} $\cA$, over a base commutative ring $k$, is a category enriched over complexes of $k$-modules; see \S\ref{sec:preliminaries}. Every (dg) $k$-algebra $A$ gives naturally rise to a dg category $\underline{A}$ with a single object and (dg) $k$-algebra of endomorphisms $A$. Another source of examples is provided by $k$-schemes since, as explained in \cite[Example~5.5]{CT1}, the derived category of perfect complexes of every quasi-compact separated $k$-scheme $X$ admits a canonical dg enhancement $\perf(X)$. 

A functor $E:\dgcat \to \cM$ defined on the category of (small) dg categories and with values in a stable Quillen model category (see \cite[\S7]{Hovey}\cite{Quillen}) is called:
\begin{itemize}
\item[(i)] {\em Morita invariant} if it sends Morita equivalences (see \S\ref{sec:preliminaries}) to weak equivalences;
\item[(ii)] {\em Continuous} if it preserves filtered (homotopy) colimits;
\item[(iii)] {\em Localizing} if it sends short exact sequences of dg categories (see \cite[\S4.6]{ICM-Keller}) to distinguished triangles
\begin{eqnarray*}
0 \to \cA \to \cB \to \cC \to 0 & \mapsto & E(\cA) \to E(\cB) \to E(\cC) \stackrel{\partial}{\to} \Sigma E(\cA)
\end{eqnarray*}
in the triangulated homotopy category $\Ho(\cM)$.
\end{itemize}
Thanks to the work of Thomason-Trobaugh, Schlichting, Keller, Blumberg-Mandell and others (see \cite{BM,Exact,Exact2,Negative,MacLane,TT}), examples of functors satisfying the above conditions (i)-(iii) include (nonconnective) algebraic $K$-theory ($K$), Hochschild homology, cyclic homology (and its variants), topological Hochschild homology, {\em etc}. As proved in {\em loc.~cit.}, when applied to $\underline{A}$ (resp. to $\perf(X)$) these functors reduce to the classical invariants of (dg) $k$-algebras (resp. of $k$-schemes). Making use of the language of Grothendieck derivators, the universal functor with respect to the above conditions (i)-(iii) was constructed in \cite[\S10]{Additive}
\begin{equation}\label{eq:universal}
U:\dgcat \to \Mot\,;
\end{equation}
in {\em loc. cit.} $U$ was denoted by $\cU_l$ and $\Mot$ by $\cM^{\mathsf{loc}}_{\dg}$. Any other functor $E:\dgcat \to \cM$ satisfying the above conditions (i)-(iii) factors through $U$ via a triangulated functor $\overline{E}:\Ho(\Mot) \to \Ho(\cM)$; see Proposition~\ref{prop:factorization}. Because of this universal property, which is reminiscent from motives, $\Ho(\Mot)$ is called the triangulated category of noncommutative motives; consult the survey article \cite{survey}. Moreover, as proved in \cite[Thm.~7.6]{CT}\cite[Thm.~15.10]{Additive}, $U(\underline{k})$ is a compact object and for every dg category $\cA$ we have the isomorphisms
\begin{eqnarray}\label{eq:Homs}
\Hom_{\Ho(\Mot)}(\Sigma^nU(\underline{k}),U(\cA))\simeq K_n(\cA) && n \in \bbZ\,.
\end{eqnarray}
Given a dg category $\cA$, an integer $n$, a functor $E:\dgcat \to \cM$, and an object $b \in \Ho(\cM)$, let us write $E_n^b(\cA)$ for the abelian group $\Hom_{\Ho(\cM)}(\Sigma^n b, E(\cA))$. For instance, when $\cA=\underline{A}$, $E=U$ and $b=U(\underline{k})$, $E_n^b(\cA)$ identifies, thanks to \eqref{eq:Homs}, with the $n^{\mathrm{th}}$ algebraic $K$-theory group $K_n(A)$ of $A$. Following Bass, a dg category $\cA$ is called {\em $E_n^b$-regular} if $E^b_n(\cA)\simeq E_n^b(\cA[t_1, \ldots, t_m])$ for all $m \geq 1$, where $\cA[t_1, \ldots, t_m]:= \cA \otimes \underline{k[t_1, \ldots, t_m]}$. Our main result, which answers affirmatively the above question, is the following:
\begin{theorem}\label{thm:main}
Let $\cA$ be a dg category, $n$ an integer, $E:\dgcat \to \cM$ a functor satisfying the above conditions (i)-(iii), and $b$ a compact object of $\Ho(\cM)$. Under these notations and assumptions, the following implication holds:
\begin{equation}\label{eq:implication-main}
\cA \,\, \mathrm{is} \,\, E_n^b\textrm{-}\mathrm{regular} \Rightarrow \cA \,\,\mathrm{is} \,\, E^b_{n-1}\textrm{-}\mathrm{regular}\,.
\end{equation}
\end{theorem}
Note that Theorem \ref{thm:main} uncovers in a direct and elegant way the three key conceptual properties (= Morita invariance, continuity, and localization) that underlie Vorst and Dayton-Weibel's implication \eqref{eq:implication}. Along its proof, we have generalized Bass' fundamental theorem and introduced a Nisnevich descent result; see Theorems \ref{thm:Nisnevich} and \ref{thm:main22}. These results are of independent interest. The above implication \eqref{eq:implication-main} shows us that regularity is preserved when $n$ is replaced by $n-1$. The same holds in the following five cases:
\begin{theorem}\label{thm:new}
Let $\cA, n, E, b$ be as in Theorem~\ref{thm:main}. 
\begin{itemize}
\item[(i)] Given an integer $i >0$, we have: $\cA$ is $E_n^b$-regular $\Rightarrow$ $\cA$ is $E^{\Sigma^{-i}b}_{n}$-regular.
\item[(ii)] Given a triangle $c \to c' \to c'' \to \Sigma c$ of compact objects in $\Ho(\cM)$, we have:
\begin{equation}\label{eq:implication22}
\cA \,\, \mathrm{is}\,\, E_n^{c'}\text{-}\mathrm{regular}\,\, \mathrm{and}\,\,E_n^{c''}\text{-}\mathrm{regular}\,\, \Rightarrow\,\, \cA\,\, \mathrm{is} \,\, E^c_n\text{-}\mathrm{regular}\,.
\end{equation}
\item[(iii)] Given a direct factor $d$ of $b$, we have: $\cA$ is $E_n^b$-regular $\Rightarrow$ $\cA$ is $E^d_{n}$-regular. 
\item[(iv)] Given a family of objects $\{c_i\}_{i\in I}$ in $\Ho(\cM)$, we have: $\cA$ is $E^{c_i}_n$-regular for every $i \in I$ $\Rightarrow$ $\cA$ is $E_n^{\oplus_{i\in I}c_i}$-regular.
\item[(v)] Consider the $k$-algebra $\Gamma$ of those $\bbN\times \bbN$-matrices $M$ which satisfy the following two conditions: (1) the set $\{M_{ij}\,|\,i,j \in \bbN \}$ is finite; (2) there exists a natural number $n_M$ such that each row and column has at most $n_M$ non-zero entries. Let $\sigma$ be the quotient of $\Gamma$ by the two-sided ideal consisting of those matrices with finitely many  non-zero entries. Under these notations, we have: $\cA$ is $E_n^b$-regular $\Rightarrow$ $\sigma(\cA):=\cA\otimes \underline{\sigma}$ is $E^b_{n+1}$-regular.
\end{itemize}
In items (iii)-(iv) the assumptions of Theorem~\ref{thm:main} are not necessary.
\end{theorem}
Roughly speaking, item (v) shows us that the converse of implication \eqref{eq:implication-main} also holds as long as on the right-hand-side one tensors $\cA$ with $\sigma$. Let us denote by $\langle \Sigma^nb |^{\natural,\oplus}$ the smallest subcategory of $\Ho(\cM)$ which contains the object $\Sigma^nb$ and which is stable under taking desuspensions, fibers of morphisms, direct factors, and arbitrary direct sums. Thanks to the above items (i)-(iv) we have: 
\begin{eqnarray}\label{eq:implication-last}
\cA \,\,\mathrm{is}\,\,E_n^{b}\text{-}\mathrm{regular} \Rightarrow \cA \,\,\mathrm{is}\,\, E^c_n\text{-}\mathrm{regular} && \forall\, c \in \langle \Sigma^n b|^{\natural,\oplus}\,.
\end{eqnarray}
Moreover, in the particular case where $\cA$ is $E_n^b$-regular for every $n \in \bbZ$ one can replace $\langle \Sigma^n b|^{\natural,\oplus}$ in the above implication \eqref{eq:implication-last} by the smallest thick localizing (=stable under arbitrary direct sums) triangulated subcategory $\langle b \rangle^{\natural,\oplus}$ of $\Ho(\cM)$ which contains $b$. Note that when $E=U$ and $b=U(\underline{k})$, \eqref{eq:implication-last} reduces to 
\begin{eqnarray}\label{eq:implication-verynew}
\cA \,\,\mathrm{is}\,\,K_n\text{-}\mathrm{regular} \Rightarrow \cA \,\,\mathrm{is}\,\, U^c_n\text{-}\mathrm{regular} && \forall\, c \in \langle \Sigma^n U(\underline{k})|^{\natural,\oplus}
\end{eqnarray}
and that in the particular case where $\cA$ is $K_n$-regular for every $n \in \bbZ$  (\eg\ $\cA=\underline{A}$ with $A$ a noetherian regular $k$-algebra) one can replace $\langle \Sigma^n U(\underline{k})|^{\natural,\oplus}$ by the triangulated category $\langle U(\underline{k})\rangle^{\natural,\oplus}$. Here is one example of the above implication \eqref{eq:implication-verynew}:
\begin{proposition}\label{prop:new}
Consider the following distinguished triangle in $\Ho(\Mot)$
\begin{equation*}
\mathrm{fib}(l) \too U(\underline{k}) \stackrel{\cdot l}{\too} U(\underline{k}) \too \Sigma \mathrm{fib}(l)\,,
\end{equation*}
where $l \geq 2$ is an integer and $\cdot l$ stands for the $l$-fold multiple of the identity morphism. Under these notations, $U_n^{\mathrm{fib}(l)}(\cA)$ identifies with Browder-Karoubi \cite{Browder} mod-$l$ algebraic $K$-theory $K_n(\cA;\bbZ/l)$. Consequently, the above implication \eqref{eq:implication-verynew} with $c:=\mathrm{fib}(l)$ reduces to: $\cA$ is $K_n$-regular $\Rightarrow$ $\cA$ is $K_n(-;\bbZ/l)$-regular.
\end{proposition}
\begin{remark}\label{rk:new2}
In the particular case where $\cA$ is a $k$-algebra $\underline{A}$ such that $1/l \in A$, Weibel proved in \cite{Weibel2,Weibel3,Weibel6} that $A$ is $K_n(-;\bbZ/l)$-regular for every $n \in \bbZ$.
\end{remark}
Intuitively speaking, Proposition~\ref{prop:new} shows us that mod-$l$ algebraic $K$-theory is the simplest replacement of algebraic $K$-theory (using fibers of morphisms) for which regularity is preserved. Many other replacements, preserving regularity, can be obtained by combining the above implication \eqref{eq:implication-verynew} with the description \eqref{eq:Homs} of the Hom-sets of the category of noncommutative motives.

Following Bass, a (quasi-compact separated) $k$-scheme $X$ is called {\em $K_n$-regular} if $K_n(X) \simeq K_n(X\times \bbA^m)$ for all $m \geq 1$, where $\bbA^1$ stands for the affine line. As mentioned above, all the invariants of $X$ can be recovered from its derived dg category of perfect complexes $\perf(X)$. Hence, let us define $E_n^b(X)$ to be the abelian group $E_n^b(\perf(X))$ and call a $k$-scheme $X$ {\em $E^b_n$-regular} if $E^b_n(X) \simeq E^b_n(X\times \bbA^m)$ for all $m \geq 1$. Making use of Theorems~\ref{thm:main} and \ref{thm:new} and of Proposition \ref{prop:new} one then obtains the following result:
\begin{theorem}\label{thm:main2}
Let $X$ be a quasi-compact separated $k$-scheme, $n$ an integer, $E:\dgcat \to \cM$ a functor satisfying the above conditions (i)-(iii), and $b$ a compact object of $\Ho(\cM)$. Under these notations and assumptions, the following implications hold:
\begin{eqnarray}\label{eq:implication2}
X \,\, \mathrm{is} \,\, E_n^b\textrm{-}\mathrm{regular} \Rightarrow X \,\,\mathrm{is} \,\, E^b_{n-1}\textrm{-}\mathrm{regular}
\end{eqnarray}
\begin{eqnarray}\label{eq:implication3}
X \,\, \mathrm{is} \,\, E_n^b\textrm{-}\mathrm{regular} \Rightarrow X \,\,\mathrm{is} \,\, E^c_{n}\textrm{-}\mathrm{regular} && \forall\, c \in \langle \Sigma^nb|^{\natural,\oplus}
\end{eqnarray}
\begin{eqnarray}\label{eq:implication4}
X \,\, \mathrm{is} \,\, K_n(-;\bbZ/l^\nu)\textrm{-}\mathrm{regular} \Rightarrow X \,\,\mathrm{is} \,\, K_{n-1}(-;\bbZ/l^\nu)\textrm{-}\mathrm{regular}\,,
\end{eqnarray}
where in \eqref{eq:implication4} $l^\nu$ is a prime power; see Thomason-Trobaugh \cite[\S9.3]{TT}.
\end{theorem}
\begin{remark}
As in the above Remark~\ref{rk:new2}, Weibel proved that in the particular case where $1/l \in \cO_X$ the $k$-scheme $X$ is $K_n(-;\bbZ/l^\nu)$-regular for every $n \in \bbZ$.
\end{remark}
When $E=U$ and $b=U(\underline{k})$, \eqref{eq:implication2} reduces to $K_n$-regularity $\Rightarrow$ $K_{n-1}$-regularity. Chuck Weibel kindly informed the author that this latter implication was proved (in a totally different way) by Corti{\~n}as-Haesemeyer-Walker-Weibel \cite[Cor.~4.4]{Cortinas2} in the particular case where $k$ is a field of characteristic zero. To the best of the author's knowledge all the remaining cases (with $k$ an arbitrary commutative ring) are new in the literature. On the other hand, \eqref{eq:implication3} reduces to the implication
\begin{eqnarray*}
X \,\, \mathrm{is} \,\, K_n\textrm{-}\mathrm{regular} \Rightarrow X \,\,\mathrm{is} \,\, U^c_{n}\textrm{-}\mathrm{regular} && \forall\, c \in \langle \Sigma^nU(\underline{k})|^{\natural,\oplus}\,.
\end{eqnarray*}
Moreover, in the particular case where $X$ is $K_n$-regular for every $n \in \bbZ$ (\eg\ $X$ a regular $k$-scheme) one can replace $\langle \Sigma^nU(\underline{k})|^{\natural,\oplus}$ by the triangulated category $\langle U(\underline{k})\rangle^{\natural,\oplus}$. Finally, to the best of the author's knowledge, implication \eqref{eq:implication4} is also new in the literature.

\begin{remark}
Theorem~\ref{thm:main} admits a ``cohomological'' analogue. Given a dg category $\cA$, an integer $n$, a functor $E:\dgcat \to \cM$, and an object $b \in \Ho(\cM)$, let us write $E^{-n}_b(\cA)$ for the abelian group $\Hom_{\Ho(\cM)}(E(\cA),\Sigma^n b)$. The dg category $\cA$ is called {\em $E_b^{-n}$-regular} if $E^{-n}_b(\cA) \simeq E^{-n}_b(\cA[t_1, \ldots, t_m])$ for all $m \geq 1$. Under these notations, the following implication 
\begin{equation}\label{eq:implication-new}
\cA \,\, \mathrm{is} \,\, E^{-n}_b\textrm{-}\mathrm{regular} \Rightarrow \cA \,\,\mathrm{is} \,\, E_b^{-n+1}\textrm{-}\mathrm{regular}
\end{equation}
holds for every functor $E$ which satisfies the above conditions (i)-(iii). Moreover, and in contrast with implication \eqref{eq:implication-main}, it is {\em not} necessary to assume that $b$ is a compact object of $\Ho(\cM)$. The proof of \eqref{eq:implication-new} is similar to the proof of \eqref{eq:implication-main}. First  replace $NE_n^b(\cA)$ by the cokernel $CE^{-n}_b(\cA)$ of the group homomorphism 
$$E^{-n}_b(\id \otimes(t=0)): E^{-n}_b(\cA) \too E_b^{-n}(\cA[t])\,,$$
then replace \eqref{eq:searched} by the group isomorphism $\mathrm{lim}\,CE_b^{-n}(\cB[x])\simeq CE_b^{-n}(\cB[x,x^{-1}])$, and finally use the new key fact that the contravariant functor $\Hom_{\Ho(\cM)}(-,\Sigma^nb)$ sends colimits to limits. 

Theorem \ref{thm:main2} also admits a ``cohomological'' analogue. In items (i)-(iv) replace $E^?_n$ by $E^{-n}_?$ and in item (v) replace the above implication by: $\cA$ is $E^{-n}_b$-regular $\Rightarrow$ $\sigma(\cA)$ is $E_b^{-n-1}$-regular. As a consequence we obtain:
\begin{eqnarray*}
\cA \,\, \mathrm{is} \,\, E_b^{-n}\textrm{-}\mathrm{regular} \Rightarrow \cA \,\,\mathrm{is} \,\, E^{-n}_c\textrm{-}\mathrm{regular} && \forall\, c \in \langle \Sigma^nb|^{\natural,\oplus}\,.
\end{eqnarray*}
In the particular case where $\cA$ is $E_b^{-n}$-regular for every $n \in \bbZ$ we can furthermore replace $\langle \Sigma^nb|^{\natural,\oplus}$ by the thick localizing triangulated category $\langle b \rangle^{\natural, \oplus}$. 
\end{remark}

\medbreak\noindent\textbf{Acknowledgments:} The author is very grateful to Denis-Charles Cisinski, Lars Hesselholt and Chuck Weibel for useful e-mail exchanges, as well as to the anonymous referee for all his comments that greatly allowed the improvement of the article.
\section{Preliminaries}\label{sec:preliminaries}
\subsection*{Dg categories}
Let $k$ be a base commutative ring and $\cC(k)$ the category of complexes of $k$-modules. A {\em differential graded (=dg) category} $\cA$ is a category enriched over $\cC(k)$ (morphism sets $\cA(x,y)$ are complexes) in such a way that composition fulfills the Leibniz rule: $d(f \circ g) =d(f) \circ g +(-1)^{\mathrm{deg}(f)}f \circ d(g)$. A {\em dg functor} $\cA \to \cB$ is a functor enriched over $\cC(k)$; consult Keller's ICM survey \cite{ICM-Keller}. In what follows we will write $\dgcat$ for the category of (small) dg categories and dg functors. 

A dg functor $\cA \to \cB$ is called a {\em Morita equivalence} if the restriction functor induces an equivalence $\cD(\cB) \stackrel{\sim}{\to}\cD(\cA)$ on derived categories; see \cite[\S3]{ICM-Keller}. The localization of $\dgcat$ with respect to the class of Morita equivalences will be denoted by $\Ho(\dgcat)$. Note that every Morita invariant functor $E:\dgcat \to \cM$ descends uniquely to $\Ho(\dgcat)$.

The tensor product of $k$-algebras extends naturally to dg categories, giving rise to a symmetric monoidal structure $-\otimes-$ on $\dgcat$ with $\otimes$-unit the dg category $\underline{k}$. As explained in \cite[\S4.2]{ICM-Keller}, this tensor product descends to a derived tensor product $-\otimes^\bbL-$ on $\Ho(\dgcat)$. Finally, recall that a dg category $\cA$ is called {\em $k$-flat} if for any two objects $x$ and $y$ the functor $\cA(x,y)\otimes-:\cC(k) \to \cC(k)$ preserves quasi-isomorphisms. In this particular case the derived tensor product $\cA\otimes^\bbL \cB$ agrees with the classical one $\cA\otimes\cB$.
\subsection*{Schemes}
Throughout this article all schemes will be quasi-compact and separated. By a $k$-scheme $X$ we mean a scheme $X$ over $\mathrm{spec}(k)$. Given a dg category $\cA$ and a $k$-scheme $X$, we will often write $\cA\otimes^\bbL X$ instead of $\cA \otimes^\bbL \perf(X)$. When $X=\mathrm{spec}(C)$ is affine we will furthermore replace $\cA \otimes^\bbL \mathrm{spec}(C)$ by $\cA\otimes^\bbL \underline{C}$. 
\subsection*{Noncommutative motives}
\begin{proposition}\label{prop:factorization}
Given a functor $E:\dgcat \to \cM$ which satisfies the above conditions (i)-(iii), there exists a triangulated functor $\overline{E}:\Ho(\Mot) \to \Ho(\cM)$ such that $\overline{E}\circ U=E$.
\end{proposition}
\begin{proof}
The category $\dgcat$ carries a (cofibrantly generated) Quillen model category whose weak equivalences are precisely the Morita equivalences; see \cite[Thm.~5.3]{IMRN}. Hence, it gives rise to a well-defined Grothendieck derivator $\HO(\dgcat)$; consult \cite[Appendix~A]{CT1} for the notion of derivator. Since by hypothesis $\cM$ is stable and the functor $E$ satisfies conditions (i)-(iii), we then obtain a well-defined localizing invariant of dg categories $\HO(E):\HO(\dgcat) \to \HO(\cM)$ in the sense of \cite[Notation~15.5]{Additive}. Thanks to the universal property of \cite[Thm.~10.5]{Additive} this localizing invariant of dg categories factors (uniquely) through $\HO(\Mot)$ via an homotopy colimit preserving morphism of derivators $\HO(\Mot) \to \HO(\cM)$. By passing to the underlying homotopy categories of this latter morphism of derivators we hence obtain the searched triangulated functor $\overline{E}:\Ho(\Mot) \to \Ho(\cM)$ which verifies $\overline{E}\circ U =E$.
\end{proof}
\section{Nisnevich descent}\label{sec:Nisnevich}
In this section we prove the following Nisnevich descent result, which is of independent interest. Its Corollary~\ref{cor:Zariski} will play a key role in the next section.
\begin{theorem}{(Nisnevich descent)}\label{thm:Nisnevich}
Consider the following (distinguished) square of $k$-schemes
\begin{equation}\label{eq:distinguished}
\xymatrix@C=2em@R=2em{
U \times_X V \ar[d] \ar[r] & V \ar[d]^-p \\
U \ar[r]_-j & X \,,
}
\end{equation}
where $j$ is an open immersion and $p$ is an {\'e}tale morphism inducing an isomorphism of reduced $k$-schemes $p^{-1}(X-U)_{\mathrm{red}} \simeq (X -U)_{\mathrm{red}}$. Then, given a dg category $\cA$ and a Morita invariant localizing functor $E:\dgcat \to \cM$, one obtains a homotopy (co)cartesian square
\begin{equation}\label{eq:square23}
\xymatrix@C=2em@R=2em{
E(\cA\otimes^\bbL X) \ar[rr]^-{E(\id \otimes^\bbL j^\ast)} \ar[d]_-{E(\id \otimes^\bbL p^\ast)} \ar@{}[drr]|{\square} && E(\cA \otimes^\bbL U) \ar[d] \\
E(\cA \otimes^\bbL V) \ar[rr] && E(\cA \otimes^\bbL (U \times_X V))
}
\end{equation}
in the homotopy category $\Ho(\cM)$; see \cite[Def.~1.4.1]{Neeman}.
\end{theorem}
\begin{proof}
Consider the following commutative diagram in $\Ho(\dgcat)$
\begin{equation*}
\xymatrix@C=2em@R=2em{
0 \ar[r] & \perf(X)_Z \ar[d]_-\sim \ar[r] & \perf(X) \ar[d]_-{p^\ast} \ar[r]^-{j ^\ast} & \perf(U) \ar[d] \ar[r] & 0  \\
0 \ar[r] & \perf(V)_{Z'} \ar[r] &\perf(V) \ar[r] & \perf(U \times_X V) \ar[r] & 0\,,
}
\end{equation*}
where $Z$ (resp. $Z'$) is the closed set $X-U$ (resp. $p^{-1}(X-U)$) and $\perf(X)_Z$ (resp. $\perf(V)_{Z'}$) the dg category of those perfect complexes of $\cO_X$-modules (resp. of $\cO_V$-modules) that are supported on $Z$ (resp. on $Z'$). As explained by Thomason-Trobaugh in \cite[\S5]{TT}, both rows are short exact sequences of dg categories; see also \cite[\S4.6]{ICM-Keller}. Furthermore, as proved in \cite[Thm.~2.6.3]{TT}, the induced dg functor $\perf(X)_Z \stackrel{\sim}{\to} \perf(V)_{Z'}$ is a Morita equivalence and hence an isomorphism in $\Ho(\dgcat)$. Following Drinfeld \cite[Prop.~1.6.3]{Drinfeld}, the functor $\cA\otimes^\bbL-: \Ho(\dgcat) \to \Ho(\dgcat)$ preserves short exact sequences of dg categories. As a consequence, we obtain the following commutative diagram in $\Ho(\dgcat)$
\begin{equation*}
\xymatrix@C=2em@R=2em{
0 \ar[r] & \cA \otimes^\bbL \perf(X)_Z \ar[d]_-\sim \ar[r] & \cA\otimes^\bbL \perf(X) \ar[d]_-{\id \otimes^\bbL p^\ast} \ar[r]^-{\id \otimes^\bbL j ^\ast} & \cA \otimes^\bbL \perf(U) \ar[d] \ar[r] & 0  \\
0 \ar[r] & \cA \otimes^\bbL \perf(V)_{Z'} \ar[r] & \cA \otimes^\bbL \perf(V) \ar[r] & \cA \otimes^\bbL \perf(U \times_X V) \ar[r] & 0\,,
}
\end{equation*}
where both rows are short exact sequences of dg categories. Recall that by hypothesis $E$ sends (in a functorial way) short exact sequences of dg categories to distinguished triangles. Consequently, by applying $E$ to the preceding commutative diagram we obtain the following morphism between distinguished triangles:
\begin{equation*}
\xymatrix@C=1.2em@R=2em{
E(\cA\otimes^\bbL \perf(X)_Z) \ar[r] \ar[d]_-\sim & E(\cA\otimes^\bbL X) \ar[d]_-{E(\id\otimes^\bbL p^\ast)} \ar[r]^-{E(\id\otimes^\bbL j^\ast)} & E(\cA \otimes^\bbL U)\ar[d] \ar[r]^-{\partial} & \ar[d]^-\sim \Sigma E(\cA\otimes^\bbL \perf(X)_Z) \\
E(\cA\otimes^\bbL \perf(V)_{Z'}) \ar[r] & E(\cA \otimes^\bbL V) \ar[r] & E(\cA \otimes^\bbL (U \times_X V)) \ar[r]_-{\partial} & \Sigma E(\cA\otimes^\bbL \perf(V)_{Z'})\,.
}
\end{equation*}
Since the outer left and right vertical maps are isomorphisms we conclude that the middle square (which agrees with the above square \eqref{eq:square23}) is homotopy (co)cartesian. This achieves the proof.
\end{proof}
\begin{corollary}{(Mayer-Vietoris for open covers)}\label{cor:Zariski}
Let $X$ be a $k$-scheme which is covered by two Zariski open subschemes $U, V \subset X$. Then, given a dg category $\cA$ and a Morita invariant localizing functor $E:\dgcat \to \cM$, one obtains a Mayer-Vietoris triangle
$$ E(\cA \otimes^\bbL X) \to E(\cA\otimes^\bbL U) \oplus E(\cA \otimes^\bbL V) \stackrel{\pm}{\to} E(\cA \otimes^\bbL (U \cap V)) \stackrel{\partial}{\to} \Sigma E(\cA \otimes^\bbL X)\,.$$
\end{corollary}
\begin{proof}
This follows from the fact that when the morphism $p$ in the square \eqref{eq:distinguished} is an open immersion, $U \times_X V$ identifies with $U \cap V$; recall also from \cite[\S1.4]{Neeman} that every homotopy (co)cartesian square has an associated distinguished ``Mayer-Vietoris'' triangle.
\end{proof}
\section{Generalized fundamental theorem}\label{sec:fundamental}
The following theorem was proved by Bass \cite[\S XII-\S 7-8]{Bass} for $n \leq 0$ and by Quillen \cite{Quillen-new} for $n \geq 1$.
\begin{theorem}{(Bass' fundamental theorem)}\label{thm:Bass}
Let $R$ be a ring and $n$ an integer. Then, we have the following exact sequence of abelian groups
$$ 0 \to K_n(R) \stackrel{\Delta}{\to} K_n(R[x])\oplus K_n(R[1/x]) \stackrel{\pm}{\to} K_n(R[x,1/x]) \stackrel{\partial_n}{\to} K_{n-1}(R) \to 0\,.$$
\end{theorem}
In this section we generalize it as follows:
\begin{theorem}{(Generalized fundamental theorem)} \label{thm:main22}
Let $\cA$ be a dg category, $n$ an integer, $E:\dgcat \to \cM$ a Morita invariant localizing functor, and $b$ and object of $\cM$. Then, we have the following exact sequence of abelian groups
\begin{equation}\label{eq:seq-main}
0\to E^b_n(\cA) \stackrel{\Delta}{\to} E_n^b(\cA[x]) \oplus E^b_n(\cA[1/x]) \stackrel{\pm}{\to} E_n^b(\cA[x,1/x])\stackrel{\partial_n}{\to} E^b_{n-1}(\cA) \to 0\,.
\end{equation}
\end{theorem}
\begin{remark}
A version of \eqref{eq:seq-main} for $k$-schemes can be found in Remark~\ref{rk:very-new}.
\end{remark}
\begin{proof}
Let $\bbP^1$ be the projective line over $\mathrm{spec}(k)$ and $i: \mathrm{spec}(k[x]) \subset \bbP^1$ and $j:\mathrm{spec}(k[1/x]) \subset \bbP^1$ its standard Zariski open cover. Since $\mathrm{spec}(k[x]) \cap \mathrm{spec}(k[1/x])= \mathrm{spec}(k[x,1/x])$, one obtains from Corollary~\ref{cor:Zariski} the following distinguished triangle 
\begin{equation}\label{eq:triangle-referee}
E(\cA \otimes^\bbL \bbP^1) \stackrel{(E(\id\otimes^\bbL i^\ast), E(\id\otimes^\bbL j^\ast))}{\too} E(\cA[x])\oplus E(\cA[1/x]) \stackrel{\pm}{\to} E(\cA[x,1/x]) \stackrel{\partial}{\to} \Sigma E(\cA\otimes^\bbL \bbP^1)\,.
\end{equation}
Note that since $k[x]$, $k[1/x]$ and $k[x,1/x]$ are all $k$-flat algebras, the derived tensor product agrees with the classical one. Let us now study the object $E(\cA \otimes^\bbL \bbP^1)$. As explained by Thomason in \cite[\S2.5-2.7]{Thomason}, we have two fully faithful dg functors 
\begin{eqnarray*}
\iota_0: \perf(\mathrm{pt}) \to \perf(\bbP^1)&& \cO_{\mathrm{pt}} \mapsto \cO_{\bbP^1}(0) \\
\iota_{-1}: \perf(\mathrm{pt}) \to \perf(\bbP^1)&& \cO_{\mathrm{pt}} \mapsto \cO_{\bbP^1}(-1) \,.
\end{eqnarray*}
Moreover, $\iota_{-1}$ induces a Morita equivalence between $\perf(\mathrm{pt})$ and Drinfeld's dg quotient $\perf(\bbP^1)/\iota_0(\perf(\mathrm{pt}))$ (see \cite[\S4.4]{ICM-Keller}). Following \cite[\S13]{Additive}, we obtain then a well-defined {\em split} short exact sequence of dg categories
\begin{equation}\label{eq:split}
\xymatrix{
0 \ar[r] & \perf(\mathrm{pt}) \ar[r]_-{\iota_0}  & \perf(\bbP^1) \ar[r]_s \ar@/_2ex/[l]_-r & \perf(\mathrm{pt}) \ar@/_2ex/[l]_-{\iota_{-1}}  \ar[r] & 0\,,
}
\end{equation}
where $r$ is the right adjoint of $\iota_0$, $r \circ \iota_0 =\id$, $\iota_{-1}$ is right adjoint of $s$, and $\iota_{-1} \circ s=\id$. As explained in the proof of Theorem~\ref{thm:Nisnevich}, the functor $\cA\otimes^\bbL-:\Ho(\dgcat) \to \Ho(\dgcat)$ preserves split short exact sequences of dg categories. Moreover, every localizing functor sends split short exact sequences to split distinguished triangles, \ie to direct sums in $\Ho(\cM)$. Therefore, by first applying $\cA\otimes^\bbL-$ to \eqref{eq:split} and then the functor $E$ we obtain the following isomorphism
\begin{equation}\label{eq:isom-1}
(E(\id\otimes^\bbL \iota_0), E(\id\otimes^\bbL \iota_{-1})): E(\cA \otimes \underline{k}) \oplus E(\cA \otimes \underline{k}) \stackrel{\sim}{\too} E(\cA \otimes^\bbL \bbP^1)\,.
\end{equation}
Recall that the line bundles $\cO_{\bbP^1}(0)$ and $\cO_{\bbP^1}(-1)$ become isomorphic when restricted to $\mathrm{spec}(k[x])$ and $\mathrm{spec}(k[1/x])$. Hence, we have the commutative diagrams
$$
\xymatrix@C=3em@R=1em{
\perf(\mathrm{pt}) \ar@/^1ex/[r]^-{\iota_0} \ar@/_1ex/[r]_{\iota_{-1}} & \perf(\bbP^1) \ar[r]^-{i^\ast} & \perf(\mathrm{spec}(k[x]))\\
\perf(\mathrm{pt}) \ar@/^1ex/[r]^-{\iota_0} \ar@/_1ex/[r]_{\iota_{-1}} & \perf(\bbP^1) \ar[r]^-{j^\ast} & \perf(\mathrm{spec}(k[1/x])) 
}
$$
and consequently we obtain the equalities:
\begin{eqnarray}
E(\id\otimes^\bbL i^\ast) \circ E(\id\otimes^\bbL \iota_0) = E(\id\otimes^\bbL i^\ast) \circ E(\id\otimes^\bbL \iota_{-1}) \label{eq:equality-1}\\
E(\id\otimes^\bbL j^\ast) \circ E(\id\otimes^\bbL \iota_0) = E(\id\otimes^\bbL j^\ast) \circ E(\id\otimes^\bbL \iota_{-1})\,. \label{eq:equality-2}
\end{eqnarray}
Now, apply Lemma~\ref{lem:general} to isomorphism \eqref{eq:isom-1} and then compose the result with $(E(\id \otimes^\bbL i^\ast), E(\id \otimes^\bbL j^\ast))$. Thanks to \eqref{eq:equality-1}-\eqref{eq:equality-2}, we obtain a morphism 
\begin{equation}\label{eq:morphism-key}
\Psi: E(\cA\otimes \underline{k}) \oplus E(\cA \otimes \underline{k}) \too E(\cA[x]) \oplus E(\cA[1/x])
\end{equation}
which is zero on the second component and
\begin{equation*}
\big(E(\id \otimes i^\ast) \circ E(\id \otimes \iota_0),E(\id \otimes j^\ast) \circ E(\id \otimes \iota_0)\big)
\end{equation*}
on the first component; note once again that since $k$, $k[x]$ and $k[1/x]$ are $k$-flat the derived tensor product agrees with the classical one. Making use of \eqref{eq:morphism-key}, the above distinguished triangle \eqref{eq:triangle-referee} identifies with
$$ E(\cA) \oplus E(\cA) \stackrel{\Psi}{\to} E(\cA[x]) \oplus E(\cA[1/x]) \stackrel{\pm}{\to} E(\cA[x,1/x])\stackrel{\partial}{\to} \Sigma E(\cA) \oplus \Sigma E(\cA)\,.$$
By applying to it the functor $\Hom_{\Ho(\cM)}(\Sigma^n b,-)$ we obtain then a long exact sequence
$$
\xymatrix@C=1em@R=0.5em{
\cdots \to E_n^b(\cA) \oplus E_n^b(\cA) \ar[r]^-{\Psi_n} & E_n^b(\cA[x]) \oplus E_n^b(\cA[1/x]) \ar[r]^-\pm & E^b_n(\cA[x,1/x]) -\!\!\!-\!\!\!-\!\!\!- \\
\stackrel{\partial_n}{\to} E_{n-1}^b(\cA) \oplus E_{n-1}^b(\cA)   \ar[r]^-{\Psi_{n-1}} & E_{n-1}^b(\cA[x]) \oplus E_{n-1}^b(\cA[1/x]) \ar[r]^-{\pm} & E_{n-1}^b(\cA[x,1/x]) \to \cdots
}
$$
As explained above, $\Psi_n$ is zero when restricted to the second component. Moreover, since the inclusions $k \subset k[x]$ and $k \subset k[1/x]$ admits canonical retractions, $\Psi_n$ is injective when restricted to the first component. This implies that the image of $\partial_n$ is precisely the second component of the direct sum. As a consequence, the above long exact sequence breaks up into the exact sequences \eqref{eq:seq-main}. This achieves the proof.
\end{proof}

\begin{lemma}\label{lem:general}
If $(f,g):A\oplus A \stackrel{\sim}{\to} B$ is an isomorphism in an additive category, then $(f,f-g):A\oplus A \stackrel{\sim}{\to} B$ is also an isomorphism.
\end{lemma}
\begin{proof}
Since $(f,g)$ is an isomorphism, there exist maps $i, h: B \to A$ such that $fi + gh =\id$, $if = \id$, $hf = 0$, $ig=0$, and $hg=\id$. Using these equalities one observes that $(i+h,-h):B \stackrel{\sim}{\to} A \oplus A$ is the inverse of $(f,f-g)$. 
\end{proof}
\begin{notation}\label{not:NE}
Given a dg category $\cA$, let us denote by $NE_n^b(\cA)$ the kernel of the surjective group homomorphism
\begin{equation}\label{eq:kernel}
E_n^b(\id \otimes (t=0)): E_n^b(\cA[t]) \stackrel{}{\too} E_n^b(\cA)\,.
\end{equation}
Note that the inclusion $k \subset k[t]$ gives rise to a direct sum decomposition $E^b_n(\cA[t]) \simeq NE_n^b(\cA) \oplus E_n^b(\cA)$. Note also that by induction on $m$, $\cA$ is $E_n^b$-regular if and only if $NE_n^b(\cA[t_1, \ldots, t_m])=0$ for all $m \geq 0$.
\end{notation}
\begin{corollary}\label{cor:N}
Under the notations and assumptions of Theorem~\ref{thm:main22}, we have the following exact sequence of abelian groups
\begin{equation*}
0 \to NE_n^b(\cA) \stackrel{\Delta}{\to} NE_n^b(\cA[x])\oplus NE_n^b(\cA[1/x]) \stackrel{\pm}{\to} NE_n^b(\cA[x,1/x]) \stackrel{\partial_n}{\to} NE_{n-1}^b(\cA) \to 0\,.
\end{equation*}
\end{corollary}
\begin{proof}
This follows automatically from the naturality of \eqref{eq:seq-main}.
\end{proof}
\section{Proof of Theorem~\ref{thm:main}}\label{sec:proof-main}
Consider the following ``substitution'' $k$-algebra homomorphism
\begin{eqnarray}\label{eq:substitution}
k[x][t] \too k[x][t] && p(x,t) \mapsto p(x,xt)\,.
\end{eqnarray}
Given a dg category $\cB$, let us denote by $\mathrm{colim}\, NE_n^b(\cB[x])$ the direct limit of the following diagram of abelian groups
$$ NE_n^b(\cB[x]) \stackrel{NE_n^b(\id\otimes \eqref{eq:substitution})}{\too} NE_n^b(\cB[x]) \stackrel{NE_n^b(\id\otimes \eqref{eq:substitution})}{\too} NE_n^b(\cB[x]) \stackrel{NE_n^b(\id\otimes \eqref{eq:substitution})}{\too} \cdots $$
We start by proving that we have a group isomorphism
\begin{equation}\label{eq:searched}
\mathrm{colim}\, NE_n^b(\cB[x]) \simeq NE_n^b(\cB[x,x^{-1}])\,.
\end{equation}
Consider first the commutative diagram
\begin{equation}\label{eq:diagram33}
\xymatrix@C=5em@R=2em{
k[x][t] \ar[d]_-{(t=0)} \ar[r]^-{\eqref{eq:substitution}} & k[x][t] \ar[d]_-{(t=0)} \ar[r]^-{\eqref{eq:substitution}} & k[x][t] \ar[d]_-{(t=0)} \ar[r]^-{\eqref{eq:substitution}} & \cdots\\
k[x] \ar@{=}[r] & k[x] \ar@{=}[r] & k[x] \ar@{=}[r] & \cdots
}
\end{equation}
Note that the colimit of the lower row is $k[x]$ while the colimit of the upper row is the $k$-algebra $R:= k[x] +t k[x,1/x][t] \subset k[x,1/x][t]$. By first tensoring \eqref{eq:diagram33} with $\cB$ and then applying the functor $E_n^b$ we obtain the commutative diagram
\begin{equation}\label{eq:diagram-big}
\xymatrix@C=5em@R=2em{
NE_n^b(\cB[x]) \ar[r]^-{NE^b_n(\id\otimes\eqref{eq:substitution})} \ar[d] & NE_n^b(\cB[x]) \ar[r]^-{NE^b_n(\id\otimes\eqref{eq:substitution})} \ar[d] & NE_n^b(\cB[x]) \ar[r]^-{NE^b_n(\id\otimes\eqref{eq:substitution})} \ar[d] & \cdots \\
E_n^b(\cB[x][t]) \ar[r]^-{E^b_n(\id\otimes\eqref{eq:substitution})} \ar[d]_-{\eqref{eq:kernel}} & E_n^b(\cB[x][t]) \ar[r]^-{E^b_n(\id\otimes\eqref{eq:substitution})} \ar[d]_-{\eqref{eq:kernel}} & E_n^b(\cB[x][t]) \ar[r]^-{E^b_n(\id\otimes\eqref{eq:substitution})} \ar[d]_-{\eqref{eq:kernel}} & \cdots \\
E_n^b(\cB[x]) \ar@{=}[r] & E_n^b(\cB[x]) \ar@{=}[r] & E_n^b(\cB[x]) \ar@{=}[r] & \cdots 
}
\end{equation}
Recall from Notation \eqref{not:NE} that each column is a (split) short exact sequence of abelian groups. The colimit of the lower row is clearly $E_n^b(\cB[x])$. Since the functors $\cB \otimes -:\dgcat \to \dgcat$ and $E:\dgcat\to \cM$ preserve filtered (homotopy) colimits and $b$ is a compact object of $\Ho(\cM)$, the colimit of the middle row identifies with $E_n^b(\cB \otimes \underline{R})$. Hence, from diagram \eqref{eq:diagram-big} one obtains the isomorphism
\begin{equation}\label{eq:aux1}
\mathrm{colim}\, NE_n^b(\cB[x]) \simeq \mathrm{Ker} \big(E_n^b(\cB \otimes \underline{R}) \stackrel{\eqref{eq:kernel}}{\too} E_n^b(\cB[x]) \big)\,.
\end{equation} 
Now, consider the $k$-algebras $R$ and $k[x]$ endowed with the sets of left denominators $S_1:=\{x^n\}_{n\geq 0} \subset R$ and $S_2:=\{x^n\}_{n\geq 0} \subset k[x]$. The $k$-algebra homomorphism 
\begin{eqnarray}\label{eq:map}
R= k[x] +tk[x,1/x][t] \too k[x] && t \mapsto 0
\end{eqnarray}
identifies $S_1$ with $S_2$ and moreover induces a quasi-isomorphism
$$
\xymatrix@C=3em@R=1.5em{
0 \ar[r] & R \ar[d]_-{\eqref{eq:map}} \ar[r] & R[S_1^{-1}]=k[x,1/x][t] \ar[d]_-{\eqref{eq:map}} \ar[r] & 0 \\
0 \ar[r] & k[x] \ar[r] & k[x][S_2^{-1}]=k[x,1/x]\ar[r] & 0\,.
}
$$
As a consequence, since $R$ and $k[x]$ are clearly $k$-flat algebras, conditions a) and b) of \cite[\S4.2]{Keller-ilc} are satisfied. In {\em loc. cit.} Keller also assumes that the base ring $k$ is coherent and of finite dimensional global dimension. However, these extra assumptions are only used to prove the localization theorem for model categories; see \cite[\S5-6]{Keller-ilc}. We obtain then a commutative diagram in $\Ho(\dgcat)$
\begin{equation}\label{eq:diagram11}
\xymatrix@C=3em@R=1.5em{
0 \ar[r] & \underline{A_1} \ar[d]_-\sim \ar[r] & \perf(R) \ar[d] \ar[r] & \perf(k[x,x^{-1}][t]) \ar[d] \ar[r] & 0 \\
0 \ar[r] & \underline{A_2} \ar[r] & \perf(k[x]) \ar[r] & \perf(k[x,x^{-1}]) \ar[r] & 0\,,
}
\end{equation}
where moreover each row is a short exact sequence of dg categories and the left vertical map is a quasi-isomorphism (and hence a Morita equivalence) of dg $k$-algebras; consult \cite[\S4.3]{Keller-ilc} for further details. By first tensoring \eqref{eq:diagram11} with $\cB$ and then applying the functor $E$ we obtain (as in the proof of Theorem~\ref{thm:Nisnevich}) a homotopy (co)cartesian square 
\begin{equation}\label{eq:diagram-square}
\xymatrix@C=3em@R=2em{
E(\cB \otimes \underline{R}) \ar[d] \ar[r] \ar@{}[dr]|{\square} & E(\cB[x,1/x][t]) \ar[d] \\
E(\cB[x]) \ar[r] & E(\cB[x,1/x])\,.
}
\end{equation}
Note that since $R$, $k[x]$, $k[x,1/x]$, and $k[x,1/x][t]$ are all $k$-flat algebras, the derived tensor product agrees with the classical one. Note also that the natural inclusions $k[x] \subset R$ and $k[x,1/x] \subset k[x,1/x][t]$ give rise to sections of the vertical maps. As a consequence, since \eqref{eq:diagram-square} is homotopy (co)cartesian, we obtain an induced isomorphism
$$ \mathrm{Ker}\big(E_n^b(\cB\otimes \underline{R}) \stackrel{\eqref{eq:kernel}}{\to} E_n^b(\cB[x]) \big) \stackrel{\sim}{\too} \mathrm{Ker}\big(E_n^b(\cB[x,1/x][t]) \stackrel{\eqref{eq:kernel}}{\to} E_n^b(\cB[x,1/x]) \big)\,.$$
Since the right-hand-side is by definition $NE_n^b(\cB[x,1/x])$ the searched isomorphism \eqref{eq:searched} follows now from isomorphism \eqref{eq:aux1}.

We are now ready to conclude the proof. As explained in Notation~\ref{not:NE}, a dg category $\cA$ is $E_n^b$-regular if and only if $NE_n^b(\cA[t_1,\ldots, t_m])=0$ for any all $m \geq 0$. Since by hypothesis $\cA$ is $E_n^b$-regular we hence have $NE^b_n(\cA[t_1, \ldots, t_m])=0$ for all $m \geq 0$. Using isomorphism \eqref{eq:searched} (with $\cB=\cA[t_1, \ldots t_{m-1}]$) we conclude that 
$$ \mathrm{colim}\, NE_n^b(\cA[t_1, \ldots, t_{m-1}][x])\simeq NE_n^b(\cA[t_1, \ldots, t_{m-1}][x, 1/x])=0\,.$$
The exact sequence of Corollary \ref{cor:N} (with $\cA=\cA[t_1, \ldots, t_{m-1}])$ implies that $NE^b_{n-1}(\cA[t_1, \ldots, t_{m-1}])=0$. Since this holds for every $m \geq 0$, we conclude finally that $\cA$ is $E_{n-1}^b$-regular. This concludes the proof of Theorem~\ref{thm:main}.
\section{Proof of Theorem \ref{thm:new}}
Item (i) follows from the combination of implication \eqref{eq:implication-main} with the equalities
$$ E_n^{\Sigma^{-i}b}(\cA) := \Hom_{\Ho(\cM)}(\Sigma^n(\Sigma^{-i}b),E(\cA))=\Hom_{\Ho(\cM)}(\Sigma^{n-i}b,E(\cA))=:E^b_{n-i}(\cA)\,.$$
In what concerns item (ii), note that by applying the bifunctor $\Hom_{\Ho(\cM)}(-,-)$ to the sequence $\Sigma^{n-1}c' \to \Sigma^{n-1}c'' \to \Sigma^nc \to \Sigma^nc' \to \Sigma^nc''$ in the first variable and to the morphism $E(\cA) \to E(\cA[t_1, \ldots, t_m])$ in the second variable, one obtains the following commutative diagram
$$
\xymatrix@C=5em@R=1.5em{
E_n^{c''}(\cA) \ar[d] \ar[r] & E_n^{c''}(\cA[t_1, \ldots, t_m]) \ar[d] \\
E_n^{c'}(\cA) \ar[d] \ar[r] & E_n^{c'}(\cA[t_1, \ldots, t_m]) \ar[d] \\
E_n^{c}(\cA) \ar[d] \ar[r] & E_n^{c}(\cA[t_1, \ldots, t_m]) \ar[d] \\
E_{n-1}^{c''}(\cA) \ar[d] \ar[r] & E_{n-1}^{c''}(\cA[t_1, \ldots, t_m]) \ar[d] \\
E_{n-1}^{c'}(\cA) \ar[r] & E_{n-1}^{c'}(\cA[t_1, \ldots, t_m])\,,
}
$$
where each column is exact. Since by hypothesis $\cA$ is $E_n^{c'}$-regular and $E_n^{c''}$-regular the two top horizontal morphisms are isomorphisms. Using implication \eqref{eq:implication-main} we conclude that the two bottom horizontal morphisms are also isomorphisms. Using the $5$-lemma one then concludes that the horizontal middle morphism is an isomorphism. This implies that $\cA$ is $E_n^c$-regular.

Let us now prove item (iii). Since by hypothesis $d$ is a direct factor of $b$, there exist morphisms $d \to b$ and $b \to d$ such that the composition $d \to b \to d$ equals the identity of $d$. This data gives naturally rise to the following commutative diagram
\begin{equation}\label{eq:retraction}
\xymatrix{
E_n^d(\cA) \ar[d] \ar[r] & E_n^b(\cA) \ar[d] \ar[r] & E_n^d(\cA) \ar[d] \\
E_n^d(\cA[t_1, \ldots, t_m]) \ar[r] & E_n^b(\cA[t_1, \ldots, t_m]) \ar[r] & E_n^d(\cA[t_1, \ldots, t_m])\,,
}
\end{equation}
where both horizontal compositions are the identity. By assumption, $\cA$ is $E_n^b$-regular and so the vertical middle morphism in \eqref{eq:retraction} is an isomorphism. From the commutativity of \eqref{eq:retraction} and the fact that isomorphisms are stable under retractions, one concludes that the vertical left-hand-side (or right-hand-side) morphism is also an isomorphism. This implies that $\cA$ is $E_n^d$-regular.

Item (iv) follow from the combination of implication \eqref{eq:implication-main} with the equalities
$$
 E_n^{\oplus_{i\in I} c_i}(\cA) :=  \Hom(\Sigma^n(\oplus_{i\in I} c_i),E(\cA))= \prod_{i\in I} \Hom(\Sigma^nc_i,E(\cA))  =:  \prod_{i \in I} E_n^{c_i}(\cA)\,,
$$
where we have removed the subscripts of $\Hom$ in order to simplify the exposition.
Let us now prove item (v). As explained in \cite[Thm.~1.2]{Suspension}, we have a canonical isomorphism $U(\sigma(\cA)) \stackrel{\sim}{\to} \Sigma U(\cA)$ in $\Ho(\Mot)$; in {\em loc. cit.} $\sigma(\cA)$ was denoted by $\Sigma(\cA)$ and $U$ by $\cU_\dg^{\mathsf{loc}}$. Hence, by applying the triangulated functor $\overline{E}$ of Proposition~\ref{prop:factorization} to the square below \eqref{eq:square-susp} (with $\cB:=\cA[t_1, \ldots, t_m]$), one obtains the square  
\begin{equation}\label{eq:square-new}
\xymatrix{
E(\sigma(\cA)) \ar[d] \ar[r]^-\sim & \Sigma E(\cA) \ar[d] \\
E(\sigma(\cA[t_1, \ldots, t_m])) \ar[r]_-\sim & \Sigma E(\cA[t_1, \ldots, t_m])
}
\end{equation}
in the homotopy category $\Ho(\cM)$. Since by construction $\sigma(\cA[t_1, \ldots, t_m])$ and $\sigma(\cA)[t_1, \ldots, t_m]$ are canonically isomorphic, \eqref{eq:square-new} gives rise to the following commutative diagram
\begin{equation}\label{eq:square11}
\xymatrix@C=2em@R=1.5em{
\Hom_{\Ho(\cM)}(\Sigma^{n+1}b,E(\sigma(\cA))) \ar[d] \ar[r]^-\sim & \Hom_{\Ho(\cM)}(\Sigma^{n+1}b,\Sigma E(\cA)) \ar[d] \\
\Hom_{\Ho(\cM)}(\Sigma^{n+1}b, E(\sigma(\cA)[t_1,\ldots, t_m])) \ar[r]_-\sim & \Hom_{\Ho(\cM)}(\Sigma^n b, \Sigma E(\cA[t_1, \ldots, t_m]))\,.
}
\end{equation}
Moreover, using the fact that $\Sigma^{-1}(-)$ is an autoequivalence of $\Ho(\cM)$, we have
\begin{equation}\label{eq:square22}
\xymatrix@C=2.5em@R=1.5em{
\Hom_{\Ho(\cM)}(\Sigma^{n+1}b, \Sigma E(\cA)) \ar[d] \ar[r]_-\sim^-{\Sigma^{-1}(-)} & \Hom_{\Ho(\cM)}(\Sigma^nb,E(\cA)) \ar[d] \\
\Hom_{\Ho(\cM)}(\Sigma^{n+1}b, \Sigma E(\cA[t_1, \ldots, t_m])) \ar[r]^-\sim_-{\Sigma^{-1}(-)} & \Hom_{\Ho(\cM)}(\Sigma^nb,E(\cA[t_1, \ldots, t_m]))\,.
}
\end{equation}
Now, recall that by hypothesis $\cA$ is $E_n^b$-regular. Hence, the vertical right-hand-side morphism in \eqref{eq:square22} is an isomorphism. Consequently, by combining \eqref{eq:square11}-\eqref{eq:square22}, we conclude that the vertical left-hand-side morphism in \eqref{eq:square11}, \ie $E_{n+1}^b(\sigma(\cA)) \to E_{n+1}^b(\sigma(\cA)[t_1, \ldots, t_m])$ is an isomorphism. This implies that $\sigma(\cA)$ is $E^b_{n+1}$-regular and so the proof is finished.
\begin{lemma}\label{lem:key-new}
Given a dg functor $F:\cA \to \cB$, we have a commutative diagram
\begin{equation}\label{eq:square-susp}
\xymatrix{
U(\sigma(\cA)) \ar[d]_-{U(\sigma(F))} \ar[r]^-{\sim} & \Sigma U(\cA) \ar[d]^-{\Sigma U(F)} \\
U(\sigma(\cB)) \ar[r]_-{\sim} & \Sigma U(\cB)
}
\end{equation}
in the homotopy category $\Ho(\Mot)$.
\end{lemma}
\begin{proof}
Thanks to \cite[Prop.~4.9]{Suspension}, we have the commutative diagram in $\Ho(\dgcat)$
\begin{equation}\label{eq:diagram}
\xymatrix{
0 \ar[r] & \cA \otimes \underline{k} \ar[d]_-{F \otimes \id} \ar[r] & \cA\otimes \underline{\Gamma} \ar[d]_-{F\otimes \id} \ar[r] & \cA\otimes \underline{\sigma} \ar[d]_-{F\otimes \id} \ar[r] & 0 \\
0 \ar[r] & \cB \otimes \underline{k} \ar[r] & \cB \otimes \underline{\Gamma} \ar[r] & \cB \otimes \underline{\sigma} \ar[r] & 0 \,,
}
\end{equation}
where both rows are short exact sequences of dg categories. Consequently, by applying the functor $U$ to \eqref{eq:diagram} we obtain the following morphism between distinguished triangles:
\begin{equation}\label{eq:diagram2}
\xymatrix{
U(\cA) \ar[d]_-{U(F)} \ar[r] & U(\cA\otimes \underline{\Gamma}) \ar[d] \ar[r] & U(\sigma(\cA)) \ar[d]_-{U(\sigma(F))} \ar[r]^-{\partial} & \Sigma E(\cA) \ar[d]^-{\Sigma U(F)} \\
U(\cB) \ar[r] & U(\cB \otimes \underline{\Gamma}) \ar[r] & U(\sigma(\cB)) \ar[r]_-{\partial}& \Sigma E(\cB)\,. 
}
\end{equation}
As explained in \cite[\S6]{Suspension}, $U(\cA\otimes \underline{\Gamma})$ and $U(\cB \otimes \underline{\Gamma})$ are isomorphic to zero in $\Ho(\Mot)$. Hence, the connecting morphisms $\partial$ are isomorphisms and so the searched commutative square \eqref{eq:square-susp} is the right-hand-side square in \eqref{eq:diagram2}. This achieves the proof.
\end{proof}
\section{Proof of Proposition \ref{prop:new}}
Consider the following distinguished triangle in $\Ho(\Mot)$
$$ U(\underline{k}) \stackrel{\cdot l}{\too} U(\underline{k}) \too U(\underline{k})/l \too \Sigma U(\underline{k})\,.$$
As proved in \cite[Prop.~2.12]{products}, one has the following isomorphisms
\begin{eqnarray*}
\Hom_{\Ho(\Mot)}(\Sigma^n(U(\underline{k})/l),U(\cA)) \simeq K_{n+1}(\cA;\bbZ/l) && n \in \bbZ\,.
\end{eqnarray*}
In {\em loc. cit.} the author worked with $k=\bbZ$ and with the additive version of $\Mot$ where localization is replaced by additivity; however, the arguments are exactly the same. The proof follows now from the fact that $U(\underline{k})/l\simeq \Sigma \mathrm{fib}(l)$ and from the definition $U_n^{\mathrm{fib}(l)}(\cA):=\Hom_{\Ho(\Mot)}(\Sigma^n\mathrm{fib}(l),U(\cA))$.
\section{Proof of Theorem \ref{thm:main2}}
Since by hypothesis $X$ is $E_n^b$-regular the isomorphism $E_n^b(X) \simeq E^b_n(X\times \bbA^m)$ holds for all $m \geq 1$. By applying Proposition~\ref{prop:schemes-monoidal} below to $X$ and to the $k$-flat $k$-scheme $Y=\bbA^m$ we obtain moreover the following isomorphisms
\begin{equation}\label{eq:isom-last}
E_n^b(X\times \bbA^m) \stackrel{\eqref{eq:isom-can}}{\simeq} E_n^b(\perf(X) \otimes \perf(\bbA^m)) \simeq E_n^b(\perf(X)[t_1,\ldots, t_m])\,.
\end{equation}
Note that since $\bbA^m=\mathrm{spec}(k[t_1,\ldots, t_m])$ is an affine $k$-flat algebra the derived tensor product agrees with the classical one. By combining \eqref{eq:isom-last} with the isomorphism $E_n^b(X)\simeq E^b_n(X\times \bbA^m)$ we conclude then that the dg category $\perf(X)$ is $E_n^b$-regular. By Theorem~\ref{thm:main} it is also $E_{n-1}^b$-regular. Hence, using again the above isomorphisms~\eqref{eq:isom-last} (with $n$ replaced by $n-1$) one concludes that the isomorphism $E_{n-1}^b(X) \simeq E_{n-1}^b(X \times \bbA^m)$ holds for all $m \geq 1$, \ie that $X$ is $E_{n-1}^b$-regular. This proves implication \eqref{eq:implication2}. Implication \eqref{eq:implication3} follows automatically from the combination of the above isomorphism \eqref{eq:isom-last} with implication \eqref{eq:implication-last}. Finally, implication \eqref{eq:implication4} follows from the combination of Proposition~\ref{prop:new} with implication \eqref{eq:implication2} and with \cite[Example~2.13]{products}. This achieves the proof.

\begin{proposition}\label{prop:schemes-monoidal}
Let $X$ and $Y$ be two quasi-compact separated $k$-schemes with $Y$ $k$-flat, $n$ an integer, $E:\dgcat \to \cM$ a Morita invariant localizing functor, and $b$ an object of $\cM$. Under these notations and assumptions, we have a canonical isomorphism
\begin{equation}\label{eq:isom-can}
E_n^b(-\boxtimes^\bbL-): E_n^b(\perf(X) \otimes^\bbL \perf(Y)) \stackrel{\sim}{\too} E_n^b(\perf(X \times Y))\,.
\end{equation}
\end{proposition}
\begin{proof}
The proof will consist on showing that the canonical maps
\begin{equation}\label{eq:can-isos}
E(-\boxtimes^\bbL -)_{Z,W} :E(\perf(Z) \otimes^\bbL \perf(W)) \too E(\perf(Z\times W))\,,
\end{equation}
parametrized by the pairs $(Z,W)$ of quasi-compact separated $k$-schemes with $W$ $k$-flat, are isomorphisms. The above isomorphism \eqref{eq:isom-can} will follow then from \eqref{eq:can-isos} (with $Z:=X$ and $W:=Y$) by applying the functor $\Hom_{\Ho(\cM)}(\Sigma^nb,-)$. Let us denote by $\mathrm{Sch}$ the category of quasi-compact separated $k$-schemes and by $\mathrm{Sch}_{\mathrm{flat}}$ the full subcategory of $k$-flat schemes. Note that we have two well-defined contravariant bifunctors
\begin{eqnarray*}
E(\perf(-)\otimes^\bbL \perf(-)) && E(\perf(-\times-))
\end{eqnarray*}
from $\mathrm{Sch}\times \mathrm{Sch}_{\mathrm{flat}}$ to $\Ho(\cM)$. Moreover, the above canonical maps \eqref{eq:can-isos} give rise to a natural transformation of bifunctors
\begin{equation}\label{eq:natural}
E(\perf(-)\otimes^\bbL \perf(-)) \Rightarrow E(\perf(-\times-))\,.
\end{equation}
Our goal is then to show that \eqref{eq:natural} is an isomorphism when evaluated at any pair $(Z,W)\in \mathrm{Sch}\times \mathrm{Sch}_{\mathrm{flat}}$. Let us start by fixing $W$. Thanks to Theorem~\ref{thm:Nisnevich} (applied to $\cA=\perf(W)$) one observes that the functor $E(\perf(-)\otimes^\bbL \perf(W))$ satisfies Nisnevich descent and hence by Corollary~\ref{cor:Zariski} Zariski descent. In what concerns $E(\perf(-\times W))$ note first that by applying the functor $-\times W$ to \eqref{eq:distinguished} one still obtains a (distinguished) square of $k$-schemes. Therefore, Theorem \ref{thm:Nisnevich} (applied to $\cA=\underline{k}$) allows us to conclude that $E(\perf(-\times W))$ satisfies also Nisnevich descent. 

Now, by the reduction principle of Bondal and Van den Bergh (see \cite[Prop.~3.3.1]{BB}) the above natural transformation \eqref{eq:natural} is an isomorphism when evaluated at the pairs $(Z,W)$, with $W$ fixed, if and only if it is an isomorphism when evaluated at the pairs $(\mathrm{spec}(C),W)$, with $C$ a commutative $k$-algebra. By fixing $Z$ and making the same argument one concludes also from the reduction principle that \eqref{eq:natural} is an isomorphism when evaluated at the pairs $(Z,W)$, with $Z$ fixed, if and only if it is an isomorphism when evaluated at the pairs $(Z,\mathrm{spec}(D))$, with $D$ a $k$-flat commutative $k$-algebra. In conclusion it suffices to show that \eqref{eq:natural} is an isomorphism when evaluated at the pairs $(\mathrm{spec}(C),\mathrm{spec}(D))$. Note that in this particular case we have the following canonical Morita equivalences
\begin{eqnarray*}
\perf(\mathrm{spec}(C))\simeq \underline{C} &\perf(\mathrm{spec}(D))\simeq \underline{D} & \perf(\mathrm{spec}(C)\times \mathrm{spec}(D))\simeq \underline{C}\otimes \underline{D}\,.
\end{eqnarray*}
Moreover, since the $k$-algebra $D$ is $k$-flat, the derived tensor product $\underline{C}\otimes^\bbL \underline{D}$ agrees with the classical one $\underline{C}\otimes \underline{D}$. By applying the functor $E$ to this latter isomorphism one obtains the evaluation $E(\underline{C}\otimes^\bbL \underline{D})\simeq E(\underline{C}\otimes\underline{D})$ of the above natural transformation \eqref{eq:natural} at the pair $(\mathrm{spec}(C),\mathrm{spec}(D))$. This concludes the proof of Proposition~\ref{prop:schemes-monoidal}.
\end{proof}

\begin{remark}\label{rk:very-new}
Given a quasi-compact separated $k$-scheme $X$, let 
\begin{eqnarray*}
X[x]:=X \times \bbA^1 & X[1/x]:= X \times \mathrm{spec}(k[1/x]) & X[x,1/x]:= X \times \mathrm{spec}(k[x,1/x])\,.
\end{eqnarray*}
Making use of Proposition~\ref{prop:schemes-monoidal} and the $k$-flatness of $k[x]$, $k[1/x]$ and $k[x,1/x]$, one observes that Theorem~\ref{thm:main22} applied to $\cA=\perf(X)$ reduces to the following exact sequence of abelian groups
$$
0 \to E_n^b(X) \stackrel{\Delta}{\to} E_n^b(X[x]) \oplus E_n^b(X[1/x]) \stackrel{\pm}{\to} E^b_n(X[x,1/x]) \stackrel{\partial_n}{\to} E_{n-1}^b(X) \to 0 \,.
$$
\end{remark}

\end{document}